\DeclareMathAlphabet{\mathpzc}{OT1}{pzc}{m}{it}

\documentclass[headsepline=true]{scrartcl}
\usepackage{amsmath}
\usepackage{amsthm, amssymb, amsfonts,bbm}
\usepackage[top=1.2in,bottom=1.2in,left=1in,right=1in]{geometry}
\usepackage[pdfborder={0 0 0}]{hyperref}
\usepackage{amscd}
\usepackage{tensor}
\usepackage{mathrsfs}
\usepackage{arydshln}
\usepackage{tikz}
\usetikzlibrary{decorations.markings}
\usetikzlibrary{decorations.pathreplacing}
\usepackage{lscape}
\usepackage{enumerate}
\usepackage[capitalize]{cleveref}
\usepackage[utf8]{inputenc}

\title{On Lehmer's question \protect\\ for integer-valued polynomials}

\author{Berend Ringeling\thanks{This work is supported by NWO grant OCENW.KLEIN.006.}
	\\[1mm]
	\small Department of Mathematics, IMAPP, Radboud University,\\[-1mm] \small PO Box 9010, 6500~GL Nijmegen, Netherlands\\[0.5mm] \small \url{b.ringeling@math.ru.nl}
}

\DeclareMathOperator{\Res}{Res}

\renewcommand{\phi}{\varphi}

\newcommand{\m}{\mathrm{m}}
\newcommand{\n}{\mathbb{N}}
\newcommand{\z}{\mathbb{Z}}

\newtheorem{thm}{Theorem}[section]
\newtheorem{lem}[thm]{Lemma} 
\newtheorem{cor}[thm]{Corollary} 
\newtheorem{prop}[thm]{Proposition}

\theoremstyle{definition} 

\newtheorem{exmp}[thm]{Example}
\newtheorem{question}[thm]{Question}

\theoremstyle{remark}

\newcommand{\ii}{\mathrm{i}}
\renewcommand{\Re}{\mathrm{Re}}

\renewcommand{\d}{\mathrm{d}}
\begin{document}
	\maketitle
\begin{abstract}
	We solve a Lehmer-type question about the Mahler measure of integer-valued polynomials.
\end{abstract}
\section{Introduction}
In the 1930s Lehmer asks, for a monic polynomial $P(x) = \prod_{j=1}^d (x - \alpha_j) \in \mathbb{Z}[x]$, whether the real quantity $\prod_{j=1}^d \max \{ 1, |\alpha_j|\}$ can be made arbitrarily close to but larger than $1$.  This quantity is called the \emph{Mahler measure of $P(x)$} \cite{BZ20}. More generally, for $P(x) = c\prod_{j=1}^d (x - \alpha_j) \in \mathbb{C}[x]$, the Mahler measure $M(P(x))$ is defined as $|c| \prod_{j = 1}^d  \max \{ 1, |\alpha_j|\}$.  Conjecturally, the answer to Lehmer's question is negative and the suspected lower bound is given by $\alpha = 1.176280818 \dots$, the unique real zero outside the unit circle of \emph{Lehmer's polynomial} $x^{10} + x^9 - x^7 - x^6 - x^5 - x^4 - x^3 + x + 1$.
Here we want to extend the original question to a bigger class of polynomials,  \emph{integer-valued} polynomials, that is, polynomials $P(x) \in \mathbb{Q}[x]$ such that $P(k) \in \z$ for all $k \in \z$. These polynomials often occur in counting problems;
 basic examples include binomial coefficients,
\[ \binom{x}{n} = \frac{x (x-1) (x-2) \dots (x - n + 1)}{n!} \in \mathbb{Q}[x]\]
for $n \in \n$.
\begin{question}
	\label{Q1}
	Can $M(P(x))$ be made arbitrarily close to but larger than $1$, when $P(x)$ is an irreducible integer-valued polynomial?
\end{question}

 The irreducibility condition is essential here, since for a reducible integer-valued polynomial $P(x)$ the bound $M(P(x)) \geq 1$ may be violated. This is seen from the example $$P(x) = \frac{x^p-x}{p}$$ for primes $p$. By Fermat's little theorem $P(x)$ is integer-valued and the Mahler measure $M(P(x)) = 1/p$ tends to $0$ as $p$ increases.
Using this example, one can construct \emph{reducible} integer-valued polynomials $P(x)$ with Mahler measure arbitrarily close to but larger than $1$. 
The following statement demonstrates that the bound $M(P(x)) \geq 1$ in Question \ref{Q1} is then best possible.
\begin{lem}
\label{Irred}
	If $P(x) \in \mathbb{Q}[x]$ is irreducible and integer-valued, then $M(P(x)) \geq 1$.
\end{lem}
In fact, one can easily construct infinitely many (non-cyclotomic) irreducible integer-valued polynomials $P(x)$ with $M(P(x))= 1$, this is demonstrated in Example \ref{exmp1} below.
A goal of this note is to answer Question \ref{Q1} in the affirmative.
To accomplish the task, we consider the family of polynomials
\[ f_p(x) = \frac{x^p - x}{p} + x^{(p+1)/2} + 1,\]
for odd integers $p$. For primes $p$, the polynomials $f_p(x)$ are integer-valued. We prove the following three statements for them.
\begin{thm}
\label{A}
	For primes $p \equiv 3 \mod 4$, $f_p(x)$ is irreducible. 
\end{thm}

\begin{thm}
\label{B}
We have the following asymptotics for $M(f_p(x))$:
\[M(f_p(x)) \sim \frac{1 + \sqrt{1 + 4/p^2}}{2}, \]
to all orders in $p$, as $p \to \infty$. In particular, $\lim_{p \to \infty} M(f_p(x)) = 1$.
\end{thm}
\begin{thm}
\label{C}
	The sequence $M(f_p(x))$ is strictly decreasing in $p$.
\end{thm}
Thus $M(f_p(x)) > 1$. Hence, the affirmative answer to Question \ref{Q1} is given by the family $f_p$ when $p \equiv 3 \mod 4$.
Here we tabulate a few values of  $M(f_p(x))$ for small primes $p$:
\begin{table}[h!]
	\begin{center}
		\label{tab:table1}
		\begin{tabular}{l|r} 
			$p$ & $M(f_p(x))$ \\
			\hline
			$3$ & $1.17503\dots$\\
			$7$ & $1.02169\dots$ \\
			$11$ & $1.00821\dots$\\
			$19$ & $1.00276\dots$ \\
		\end{tabular}
	\end{center}
\end{table}

\section{Properties for the Mahler measure of integer-valued polynomials}
\begin{proof}[Proof of Lemma \ref{Irred}.]
The Mahler measure of any polynomial $P(x) \in \mathbb{Q}[x]$ is bounded from below by the absolute value of the constant term. Indeed if 
\[ P(x) = a_d x^d + \dots + a_0 = a_d \prod_{j=1}^d (x - \alpha_j),\]
then
\[M(P(x)) = |a_d| \prod_{j=1}^d \max \{1, |\alpha_j|\} \geq |a_d| \prod_{j=1}^d |\alpha_j| = |a_d| \frac{|a_0|}{|a_d|} = |a_0|. \]
Moreover, if $P(x)$ is \emph{integer-valued} and \emph{irreducible}, $P(0)$ is guaranteed to be a non-zero integer. Hence $M(P(x)) \geq 1$. \qedhere 
\end{proof}
The following example shows that one can find infinitely many non-cyclotomic irreducible integer-valued polynomials with Mahler measure exactly $1$.

\begin{exmp}
\label{exmp1}
Consider the integer-valued polynomial
\[g_p(x) = \frac{x^p -x}{p} + 1\]
for primes $p > 2$. The zeros of $g_p(x)$ all lie outside the complex unit circle, as otherwise for any zeros $\alpha$ of $g_p$ inside or on the unit circle, we would have the contradictary inequality
\[0 = |g_p(\alpha)| = \left|1 +\frac{\alpha^p}{p}-\frac{\alpha}{p}\right| \geq  1 - \frac{2}{p} > 0.\]
As a consequence of this, we find $M(g_p(x)) = 1$.

We want to show that the polynomial $p g_p(x)$ is irreducible. If it were reducible, then at least one of the irreducible factors would have constant term $1$; this is impossible since $g_p$ has all the zeros outside the unit circle. 
Thus, we have found an infinite family of (non-cyclotomic) irreducible integer-valued polynomials.
\end{exmp}

\section{Irreducibility}
\begin{proof}[Proof of Theorem \ref{A}]
For a polynomial $P(x)$ of degree $d$, write $\widetilde{P}(x) = x^d P(1/x)$ for its reciprocal.
We prove the irreducibility of the polynomials $f_p$ for primes $p \equiv 3\mod 4$ following the method first used by Ljunggren in \cite{Ljunggren}, also see the expository notes \cite{Conr}. The irreducibility of $f_3$ and $f_7$ is immediate, so we deal with $p>7$ from now on.

Write \begin{equation*}
	f^{*}_p(x) = p f_p(x) = x^p + p x^{\frac{p+1}{2}} - x + p
\end{equation*}
and
\begin{equation*}
	\widetilde{f^{*}_p}(x) = x^p f^{*}_p(1/x) = px^p - x^{p-1} + px^{\frac{p-1}{2}} + 1
\end{equation*}
for its reciprocal.
\begin{lem}
\label{lem21}
The polynomials $f^{*}_p$ and $\widetilde{f^{*}_p}$ have no zeros in common.
\end{lem}
\begin{proof}
	Suppose $\alpha$ is a zero of both $f^{*}_p$ and $\widetilde{f^{*}_p}$,
so that
	\begin{equation*}
		\alpha^p - \alpha + p \alpha^{\frac{p+1}{2}} + p =0 \quad \text{and} \quad	\alpha - \alpha^p + p \alpha^{\frac{p+1}{2}} + p\alpha^{p+1}=0.
	\end{equation*}
The equations imply
	\[ (\alpha^{\frac{p+1}{2}} + 1)^2=0,\]
	hence $\alpha^{\frac{p+1}{2}} = - 1$.
	Substituting this in the first equation we find that $\alpha = \pm 1$. This is impossible if $p \equiv 3 \mod 4$.
\end{proof}

Suppose $f^{*}_p(x)$ is reducible, i.e. $f^{*}_p= g h$ for $g, h \in 
\mathbb{Z}[x]$ of positive degree. Define an auxiliary polynomial $$k = g \widetilde{h} = b_p x^p + \dots + b_0;$$
then $k \widetilde{k} = f^{*}_p \widetilde{f^{*}_p}$.
Note that $k \neq \pm f^{*}_p$ or $\pm \widetilde{f^{*}_p}$, as otherwise $k = g \widetilde{h}$ and $f^{*}_p = g h$ are equal, up to sign, hence $\widetilde{h}$ and $h$ share a common zero, which is impossible by Lemma \ref{lem21}.

We next compute the coefficients of $k$ by comparing the coefficients in

\begin{equation}
\label{EQ1}
	k \widetilde k = f^{*}_p \widetilde{f^{*}_p} = px^{2p} - x^{2p-1} + p^2 x^{\frac{3p+1}{2}} - px^{p+1} + 2(p^2+1)x^p - px^{p-1} + p^2 x^{\frac{p-1}{2}} -x + p.
\end{equation}
Reading off the $x^{2p}$-coefficient we have $b_0 b_p = p$. We can assume that $b_0 = \pm p$ and $b_p = \pm 1$, possibly by interchanging $k$ and $\widetilde{k}$. Further we may assume $b_0 = p$ and $b_p = 1$ by possibly replacing $k$ with $-k$.

Comparing the $x^p$-coefficient of $k \widetilde{k}$ in \eqref{EQ1}, we find
\begin{equation*}
	2(p^2+1) = b_0^2 + \dots + b_p^2,
\end{equation*}
therefore
\begin{equation}
	\label{This}
	p^2+ 1 = b_1^2 + \dots + b_{p-1}^2.
\end{equation}
Comparing the $x$-coefficient in \eqref{EQ1} we conclude that
\begin{equation*}
	-1 = b_0 b_{p-1} + b_1 b_p
\end{equation*}
implying $b_1 = - 1  - b_{p-1} p$. The latter equality is only possible if either $b_{p-1} = 0$ and $b_1 = -1$, or $b_{p-1} = -1$ and $b_1 = p -1$, as otherwise \eqref{This} fails. Consider the two cases separately.

\textbf{Case $b_{p-1} = -1$ and $b_1 = p-1$.}
We obtain from \eqref{This}
\begin{equation}
	\label{This3}
	2p - 1 = b_2^2 + \dots + b_{p-2}^2
\end{equation}
Compare the $x^2$-coefficient to find that
\[ b_0b_{p-2} + b_1b_{p-1} + b_2b_{p} = 0, \]
so that $b_2 = -1 - p(b_{p-2} - 1)$. According to \eqref{This3}, the equality is only possible if $b_{p-2} = 1$ and $b_2 = -1$. 
Next compare the $x^3$-coefficient to find that
\[ b_0b_{p-3} + b_1b_{p-2} + b_2b_{p-1} + b_3b_{p} = 0,\]
hence $p(b_{p-3} + 1) + b_3 = 0$. Again, from \eqref{This3} we conclude that $b_{p-3} = -1$ and $b_3 = 0$.
We claim that $b_{p-j} = (-1)^j$ and $b_j = 0$ for $2 < j < \frac{p-1}{2}$. Comparing the $x^j$-coefficient for such $j$ gives

\[b_0b_{p-j} + b_1 b_{p-j+1} + b_2 b_{p-j+2} + \dots + b_j b_p =0. \]
By induction all the terms $b_{i}$ vanish for $2 < i < j$,
so that $b_j = -p(b_{p-j} - (-1)^j)$. From \eqref{This3} and the fact that $p$ divides $b_j$, we conclude that $b_j = 0$ and $b_{p-j} = (-1)^j$.
Finally, compare the coefficient of $x^{\frac{p-1}{2}}$ in \eqref{EQ1}:
\[p^2 = b_0b_{\frac{p+1}{2}} + b_1b_{\frac{p+1}{2} + 1} + b_2 b_{\frac{p+1}{2} + 2} + \dots + b_{\frac{p-1}{2}} b_p. \]
This translates into
\[p^2 = p (b_{\frac{p+1}{2}} - (-1)^{\frac{p-3}{2}}) + b_{\frac{p-1}{2}}. \]
Therefore, $b_{\frac{p-1}{2}}$ is divisible by $p$, hence  $b_{\frac{p-1}{2}} = 0$ from \eqref{This3} implying $b_{\frac{p+1}{2}} = (-1)^{\frac{p-1}{2}} + p = p - 1$. This calculation contradicts \eqref{This3}.

\textbf{Case $b_{p-1} = 0$ and $b_1 = -1$.} In this case we have
\begin{equation}
	\label{This4}
	p^2 = b_2^2 + \dots + b_{p-2}^2.
\end{equation}
We claim that $b_j = 0$ for $1 < j < \frac{p-1}{2}$.
Suppose otherwise, let $1< j' <\frac{p-1}{2}$ be the smallest integer such that $b_{j'} \neq 0$. 
Comparing the $x^i$-coefficient for $1 < i < j'$ in \eqref{EQ1} results in
\[b_0 b_{p-i} + b_1 b_{p-i+1} \dots + b_ib_p =0;\]
it follows by induction that $b_{p-i} = 0$ for all such $i$ as well.

Comparing the $x^{j'}$-coefficient in \eqref{EQ1} we find out that
\[ b_0 b_{p - j'} + b_1b_{p-j' + 1} + \dots+ b_{j'}b_p = 0,\]
hence $p b_{p-j'} + b_{j'} = 0$. Since $b_{j'} \neq 0$ by our assumption, we have $|b_{j'}| \geq p$ and $|b_{p -j'}| \geq 1$.  Comparing this with \eqref{This4} we find this impossible. The contradiction implies that $b_j =0$ and $b_{p-j} = 0$ for $1 < j < \frac{p-1}{2}$.

Finally, consider the $x^{\frac{p+1}{2}}$-coefficient in \eqref{EQ1}:
\[b_0 b_{p - \frac{p-1}{2}} + b_1 b_{p - \frac{p-1}{2} + 1} + \dots + b_{\frac{p-1}{2}} b_p = p^2; \]
this simplifies to 
\[p b_{\frac{p+1}{2}} + b_{\frac{p-1}{2}} = p^2. \]
Comparing with \eqref{This4}, the only solution to this equation is $b_{\frac{p+1}{2}} = p$ and $b_{\frac{p-1}{2}} = 0$. 
We conclude that $k = f^{*}_p$, which gives a contradiction.

Thus, $f^{*}_p$ is irreducible. This proves Theorem \ref{A}.
\end{proof}
\section{Asymptotics}
For this part, it is more convenient to work with the \emph{logarithmic Mahler measure} $\m(P(x)) = \log(M(P(x)))$.
Jensen's formula allows one to write it as
\begin{equation}
	\label{Jensen}
	\m(P(x)) = \frac{1}{2 \pi \ii} \oint_{|z| = 1} \log |P(z)| \ \frac{\d z}{z}.
\end{equation}

Denote $N = (p-1)/2$ and $Q_p(x) = (x^2-1)/p+x$ and define
\[ m_p = \m \left( \frac{x^p - x}{p} + x^{\frac{p+1}{2}} + 1 \right) = \m(xQ_p(x^N) + 1).\]
We will show that, for all integers $N$, 
\[m_p \sim \m (x Q_p(x^N)) = \m(Q_p(x)) = \log  \frac{1 + \sqrt{1 + 4/p^2}}{2}  \]
to all orders in $p$, i.e. the difference of $m_p$ and $\m(Q_p(x))$ is $\mathcal{O}(p^n)$ for all $n \in \z$.

We have
\begin{align} 
\label{mp}
	m_p - \m \left( x Q_p(x^N) \right) = \m\left( 1 + \frac{1}{x Q_p(x^N)}  \right) = \frac{1}{N} \cdot \m \left( 1 + \frac{(-1)^{N+1}}{x Q_p(x)^N} \right),
\end{align}
where the last equality follows from the more general observation:
\begin{lem}
	\label{ZudLem}
	If $P(x)$ is a polynomial and $N$ an integer, then
	\[\m \left(1 + \frac{(-1)^{N+1}}{x P(x)^N}\right) = N \cdot \m \left(1 + \frac{1}{x P(x^N)}\right) \] 
\end{lem}
\begin{proof}
Indeed, Jensen's formula implies that
	\begin{align*}
		\m \left(1 + \frac{(-1)^{N+1}}{x P(x)^N}\right) = \m \left(1 + \frac{(-1)^{N+1}}{x^N P(x^N)^N}\right) &= \m\left(1 - \left( \frac{-1}{xP(x^N)} \right)^N \right) \\&= \sum_{\xi \colon \xi^N = 1} \m \left(1 + \frac{\xi}{x P(x^N)}\right),
	\end{align*}
	where the sum is over all roots of unity of degree $N$.
	The required identity follows from noticing that  \[\m \left(1 + \frac{\xi}{x P(x^N)}\right) =\m \left(1 + \frac{1}{x P(x^N)}\right),\] by substituting $\xi x$ for $x$ in the integral \eqref{Jensen} for the corresponding Mahler measure.
\end{proof}
Since 
\[\frac{1}{|Q_p(z)|^2} = \left|\frac{p}{z^2 + pz -1} \right|^2 = \frac{p^2}{2 + p^2 - 2 \, \Re(z^2)} < 1\]
for $z \in \mathbb{C} \setminus \{ \pm 1\}, |z| = 1$, we get the convergent expansion
\begin{align*}
	\log\left(1 + \frac{(-1)^{N+1}}{z Q_p(z)^N} \right) = \sum_{\ell = 1}^\infty \frac{(-1)^{\ell N-1}}{\ell z^\ell Q_p(z)^{\ell N}}.
\end{align*}
for all such $z$.
From this we find out that
\begin{equation}
	\label{eqn}
	\m \left(1 + \frac{(-1)^{N+1}}{x Q_p(x)^N} \right) = \Re \, \frac{1}{2 \pi \ii} \oint_{|z| = 1} \log \left(1 +\frac{(-1)^{N+1}}{z Q_p(z)^N}\right) \frac{\d z}{z} = \Re \, \sum_{\ell=1}^\infty \frac{(-1)^{\ell N-1}}{\ell} F_{\ell},
\end{equation} 
where 
\[F_{\ell} = \frac{1}{2 \pi \ii} \oint_{|z| = 1} \frac{\d z}{z^{\ell+1} Q_p(z)^{\ell N}} = \frac{p^{\ell N}}{2 \pi \ii} \oint_{|z| = 1} \frac{\d z}{z^{\ell+1} (z-\alpha_1)^{\ell N}(z-\alpha_2)^{\ell N}}\]
and $\alpha_1,\alpha_2$ are the zeros of $Q_p(x)$ ordered by $|\alpha_2| > 1 > |\alpha_1|$.
We will examine the asymptotics of $F_{\ell}$ for $\ell \geq 1$ as $p \to \infty$. We can explicitly compute these integrals.
\begin{lem}
\label{lem23}
	For $\ell \geq 1$, we have
\begin{equation*}
 F_\ell = (-1)^\ell p^{\ell N} \sum_{j = 0}^{\ell N - 1} \binom{2\ell N-2-j}{\ell N -1} \binom{\ell+j}{j} \frac{1}{(\alpha_2-\alpha_1)^{2 \ell N-1-j}\alpha_2^{\ell+1+j}}.
\end{equation*}
\end{lem}
\begin{proof}
	This follows from Cauchy's integral theorem. The integrand has precisely one singularity outside the unit circle. Therefore, the value of the integral is given by $$-\Res_{z = \alpha_2} \frac{1}{z^{\ell+1}Q_p(z)^{\ell N}}.$$
	The formula follows by expanding $1/z^{\ell+1}$ into a series in $z - \alpha_2$:
	\begin{equation}
		\label{SER}
		\frac{1}{z^{\ell+1}} = \sum_{j = 0}^\infty (-1)^j \binom{\ell+j}{j} \frac{1}{\alpha_2^{\ell+1+j}} (z - \alpha_2)^j
	\end{equation}
and extracting the nonpositive powers of $z-\alpha_2$ in the Laurent expansion of $1/Q_p(z)^{\ell N}$:
\begin{equation}
\label{PF}
	\frac{1}{(z-\alpha_1)^{\ell N} (z - \alpha_2)^{\ell N}} = \sum_{j = 0}^{\ell N} (-1)^{j} \binom{\ell N + j -1}{j}\frac{1}{(\alpha_2 - \alpha_1)^{\ell N + j}} (z-\alpha_2)^{j - \ell N} + \mathcal{O}(z -\alpha_2).
\end{equation}
Taking the product of \eqref{SER} and \eqref{PF} we conclude with the formula
\begin{equation*}
	(-1)^{\ell-1}p^{\ell N}\sum_{j = 0}^{\ell N - 1} \binom{2\ell N-2-j}{\ell N -1} \binom{\ell+j}{j} \frac{1}{(\alpha_2-\alpha_1)^{2\ell N-1-j}\alpha_2^{\ell+1+j}}
\end{equation*}
for the coefficient of $1/(z-\alpha_2)$. \qedhere
\end{proof}
Using Lemma \ref{lem23}, we will estimate $|F_\ell|$ from above.
\begin{lem}
	\label{Lem24}
	For $\ell \geq 1$, we have
	\[|F_\ell| \leq \frac{1}{p^{\ell(N+1)}} \binom{2\ell N + \ell - 1}{\ell N}.\]
\end{lem}
\begin{proof}
	The estimates $|\alpha_2 - \alpha_1| \geq p$ and $|\alpha_2| \geq p$ imply
	\begin{align*}
		|F_\ell| &\leq \frac{1}{p^{\ell (N+1)}} \sum_{j = 0}^{\ell N-1} \binom{2 \ell N-2-j}{\ell N-1} \binom{\ell+j}{j}\\
		&= \frac{p-1}{p+1} \frac{1}{p^{\ell (N+1)}} \binom{2\ell N+\ell-1}{\ell N} \leq \frac{1}{p^{\ell (N+1)}} \binom{2\ell N+\ell-1}{\ell N}. \qedhere
	\end{align*}
\end{proof}
It follows from Lemma \ref{Lem24} that $F_\ell$ decays exponentially in $\ell N$.
\begin{proof}[Proof of Theorem \ref{B}.]
Using Equations \eqref{mp}, \eqref{eqn} and Lemma \ref{Lem24}, we conclude that
\[\left|m_p - \m(Q_p(x))\right| \leq \frac{1}{p^{N+1}} \binom{p-1}{N} =: \epsilon_p\]
meaning that the difference of the Mahler measures decays exponentially in $p$ as $p \to \infty$. This finishes the proof of Theorem \ref{B}.
\end{proof}
\begin{proof}[Proof of Theorem \ref{C}.]
To show that the sequence $m_p$ for odd $p$ is decreasing, it suffices to prove the inequality
\begin{align}
	\label{Ineq}
	\m(Q_p(x)) - \epsilon_{p} > \m(Q_{p+2}(x)) + \epsilon_{p+2},
\end{align}
where $\epsilon_{p}$ is defined in the proof of Theorem \ref{B}. 
We can estimate  $\m(Q_p(x)) - \m(Q_{p+2}(x))$ from below using that $\log(x) > 1 - \frac{1}{x}$ for $x > 1$. Indeed, for $p \geq 5$ we have
\begin{align*}
	\m(Q_p(x)) - \m(Q_{p+2}(x)) &= \log \frac{1+\sqrt{1+4/p^2}}{1+\sqrt{1+4/(p+2)^2}}\\
	& > \frac{\sqrt{1+4/p^2} - \sqrt{1+4/(p+2)^2}}{1+\sqrt{1+4/p^2}}\\
	& \geq \frac{1}{2p^2} - \frac{1}{2(p+2)^2} \geq \frac{1}{p^3}.
\end{align*}
On the other hand, using $\binom{2n}{n} \leq 2^n$ for $n \geq 1$, we can estimate $\epsilon_p + \epsilon_{p+2}$ from above: for $p \geq 7$ we obtain
\begin{align*}
	\epsilon_{p}+\epsilon_{p+2} &\leq \frac{1}{p^{N+1}} 4^N + \frac{1}{(p+2)^{N+2}} 4^{N+1}
	\leq \left( \frac{4}{p} \right)^{N+1} \leq \left( \frac{3}{4} \right)^p \leq \frac{1}{p^3}.
\end{align*}
This implies inequality \eqref{Ineq} for $p \geq 7$. Together with
\[ m_3 =  0.16129\dots \, , \qquad m_5 =0.04920\dots \, , \qquad m_7 =0.02145\dots \, ,\]
it concludes our proof of Theorem \ref{C}.
\end{proof}
\section{Discussion}

The choice for the family of polynomials $f_p(x)$ is far from optimal: among integer-valued polynomials of prime degree $p \equiv 3 \mod 4$, it is not the one with smallest Mahler measure larger than $1$. This can already be seen when $p=3$: an integer-valued polynomial with the smallest Mahler measure is
$$Q_3(x) = \frac{2}{3}x^3 - \frac{1}{2}x^2 - \frac{1}{6}x -1,$$ with the Mahler measure $1.02833 \dots$ much smaller than $M(f_3(x)) = 1.17503\dots$\,.

For $d = 2, 3, \dots$, define $Q_d(x)$ to be an irreducible integer-valued polynomial of degree $d$ with smallest Mahler measure larger than $1$. Then the following questions arise.
\begin{question}
	How to (efficiently) compute these polynomials $Q_d(x)$?
\end{question}
\begin{question}
 What can be said about the asymptotics of $M(Q_d(x))$ for $d \to \infty$?
\end{question}
\paragraph{Acknowledgements}
The author would like to thank Fran\c{c}ois Brunault, David Hokken and Wadim Zudilin for interesting discussions and helpful comments.

\end{document}